\documentclass[11pt, reqno]{amsart}
\usepackage{amsmath,amsthm,amscd,amssymb,amsfonts, amsbsy}
\usepackage{latexsym, color, enumerate}
\usepackage{pxfonts}
\usepackage{mathrsfs}
\usepackage{marginnote}
\usepackage{todonotes}
\usepackage[colorlinks=true, pdfstartview=FitV, linkcolor=blue, citecolor=blue, urlcolor=blue]{hyperref}
\usepackage{mathtools}
\mathtoolsset{showonlyrefs}
\usepackage[margin=1 in]{geometry}
\usepackage{comment}

\allowdisplaybreaks[4]

\usepackage[numbers,sort]{natbib}

\newtheorem{innercustomgeneric}{\customgenericname}
\providecommand{\customgenericname}{}
\newcommand{\newcustomtheorem}[2]{%
  \newenvironment{#1}[1]
  {%
   \renewcommand\customgenericname{#2}%
   \renewcommand\theinnercustomgeneric{##1}%
   \innercustomgeneric
  }
  {\endinnercustomgeneric}
}

\newcustomtheorem{customthm}{Theorem}

\usepackage{tikz}
\usetikzlibrary{snakes}
\usepackage{multicol}
\usepackage{cancel}

\numberwithin{equation}{section}

\theoremstyle{plain}
\newtheorem{theorem}{Theorem}[section]
\newtheorem{lemma}[theorem]{Lemma}

\theoremstyle{definition}
\newtheorem{definition}[theorem]{Definition}

\newtheorem{remark}[theorem]{Remark}

\newtheorem*{theorem*}{Theorem}

\newcommand{\la}{\lambda}
\newcommand{\R}{\mathbb{R}}

\newcommand{\re}{\mathrm{Re}}
\newcommand{\im}{\mathrm{Im}}

\newcommand\ef{e^{-\lambda f}u}

\newcommand\blfootnote[1]{%
  \begingroup
  \renewcommand\thefootnote{}\footnote{#1}%
  \addtocounter{footnote}{-1}%
  \endgroup
}

\def\XXint#1#2#3{{\setbox0=\hbox{$#1{#2#3}{\int}$}
		\vcenter{\hbox{$#2#3$}}\kern-.5\wd0}}

\begin{document}

\title[]{Carleman estimates for third order operators of KdV and non KdV-type and applications}

\author[Serena Federico]{Serena Federico}

\address{Serena Federico
\newline \indent Dipartimento di Matematica, Universit\`a di Bologna
\newline \indent Piazza di Porta San Donato 5, 40126, Bologna, Italy}
\email{serena.federico2@unibo.it}


\begin{abstract}
In this paper we study a class of variable coefficient third order partial differential operators on $\R^{n+1}$, containing, as a subclass, some variable coefficient operators of KdV-type in any space dimension. For such a class, as well as for the adjoint class, we obtain a Carleman estimate and the local solvability at any point of $\R^{n+1}$. A discussion of possible applications in the context of dispersive equations is provided.
\end{abstract}

\maketitle

\setcounter{tocdepth}{1}
\tableofcontents

\parindent = 10pt     
\parskip = 8pt

\section{Introduction}
\blfootnote{The author has no conflict of interest to disclose.}
In this paper, we will continue the investigation of some variable coefficient PDOs (partial differential operators) built from a  system of real smooth vector fields, initiated in the works \cite{FP2}, \cite{FP3}, \cite{F2} and \cite{F4} (see also \cite{FP1}, \cite{F}, \cite{F3}). Since when the celebrated works by Kolmogorov \cite{Ko} first, and by H\"ormander \cite{Ho2} afterwards, about the hypoellipticity of operators written as sums of squares of vector fields were published, a lot of connected problems have been investigated. In particular, H\"ormander's hypoelliptic theorem in  \cite{Ho2} opened up the study of sub-Laplacians on Lie groups (see \cite{RS}), of parametrices for such operators, of sharp estimates in appropriate functional spaces (see the pioneering works \cite{FoSt} and \cite{RS}), but also, among other questions, that of the local solvability of closely related models (see, for instance, \cite{LR}, \cite{MPR}, \cite{FP2}, \cite{FP3}, \cite{F2},\cite{F4} and references therein).

In the series of works \cite{FP2}, \cite{FP3}, \cite{F2} and \cite{F4}, the authors focus on the local solvability of some classes of degenerate second order PDOs built from a system of real smooth vector fields, plus a somewhere vanishing function. The use of the vanishing function is twofold: on one side it adds degeneracy to a model which is already degenerate by itself, on the other side it permits to include in the treatment some operators like the Kannai operator, that is operators having a changing sign principal symbol. The introduction of the vanishing function requires a price to pay, price which is given in terms of conditions on the lower order terms of the operators (see \cite{FP1}, \cite{F}, \cite{F3} for an overview of such results). Nevertheless, the classes studied in these works encompass different lower order terms, which makes it possible to include parabolic-type operator, Schr\"odinger-type operators, a blend of the two, and some prototypes with non-smooth coefficients.

What we do in the present paper, is to take inspiration by these works and by the form of some operators relevant in physics, and build a class of operators of order three (instead of two) starting from a system of smooth real vector fields. Let us say right away that we will not involve any additional somewhere vanishing function, since the model under consideration is already highly degenerate and complicated without this adjustment. 

Besides the connection with the aforementioned previous works, the interest for such kind of operators has several motivations (see \cite{KS} and references therein). One of them is represented by the applications to variable coefficient dispersive equations of KdV-type, where we refer to KdV-type operators as those of the form $i\partial_t+\mathcal{L}(t,x,D_x)$, with $\mathcal{L}$ being a third order PDO with smooth coefficients. \\
\indent In the last decades variable coefficient Schr\"odinger equations have attracted lots of attention. Smoothing and Strichartz estimates have been proved under different hypotheses (see \cite{CS},\cite{ST}, \cite{MMT}, \cite{KPRV}, \cite{FS1}, \cite{FS0}, \cite{FR}, \cite{F0}, \cite{RZ} and references therein), nonlinear problems have been solved (see \cite{MMT}, \cite{KPRV}, \cite{FS1}, \cite{FS0}, \cite{FR}, \cite{F0}), and uniqueness results have been proved (see \cite{EKPV1}, \cite{EKPV3}, \cite{EKPV4} and \cite{KPPV} in the constant coefficient case, and \cite{CF},\cite{FLY} and references therein for variable coefficient cases).
For KdV-type equations the investigation has not been pushed that far, possibly because of the unknown real analogue of this equation in dimensions higher than two, and also because variable coefficient third order equations can be much more challenging to study. Here we wish to start a local analysis of some variable coefficient cases, and prove Carleman type estimates to be employed to reach uniqueness results. Let us underline that such estimates have a pivotal role in establishing uniqueness properties, not only at a local (in space) level, but also for global results. In the context of dispersive equations, results for KdV and ZK (Zakharov-Kuznetsov) equations can be found, for instance, in \cite{CFL},\cite{EKPV2} and \cite{Z}, while for nonlocal operators we refer the interested reader to  \cite{KPPV} and \cite{RSV}.  Note that all these cases deal with constant coefficients operators, while here we are concerned with the situation where the coefficients are variable. As far as the author knows, there are no results in the space variable coefficients setting. Recently, some results on $p$-evolution equations in one space dimension have appeared in \cite{AJAC}. In that work the operators under investigation have time-dependent coefficients in the leading part, and the class only includes KdV operators (since the space domain is one dimensional) with time variable coefficients in the leading part, but not ZK operators. Here we are concerned with space-dependent coefficients in any space dimension.\\
\indent An other motivation driving this work, is that of the local solvability problem for multiple characteristics PDOs. This problem is very hard to attack, and general results often require precise geometrical conditions on the characteristic set. We refer the interested reader to \cite{T} for a nice overview of the classical local solvability problem, and to \cite{D} for the resolution of the Nirenberg-Treves local solvability conjecture.

Let us give a description of the family of operators we will be considering in this paper.

 Given a system of vector fields $\{iX_j(x,D)\}_{j=0}^N$  on  $\mathbb{R}^{n+1}$ – where $N$ is a positive integer not necessarily equal to $n$ or $n+1$, and $D=(D_1,\ldots, D_{n+1})=(-i\partial_{x_1}, \ldots, -i\partial_{x_{n+1}})$ – such that
 
\begin{itemize}
\item the vector fields $iX_j(x,D)$, $j=0,\ldots,N$, are \textit{smooth} and \textit{real};
\item the vector fields $iX_j(x,D)$; $j=1,\ldots,N$, form a global involutive distribution, that is, $[X_l,X_m](x,D)=-[iX_l,iX_m]$ is uniquely determined by a linear combination with smooth coefficients of the vector fields $\{iX_j(x,D)\}_{j=1,\ldots,N}$, for all $l,m=1,\ldots, N,$  and for all $x\in \R^{n+1}$;
\item $X_1$ is nondegenerate, namely it is nowhere vanishing on $\R^{n+1}$;
\end{itemize}
we define $P_1$ as the third order PDO of the form

\begin{align}
P_1(x,D)&:=X_1\sum_{j=1}^N X_j^* X_j+X_0.\label{eq.P1}
\end{align}
The latter is the precise class of operators we will be dealing with,
which, depending on $N$ and on the choice of the system of vector fields, includes (linear) KdV-type operators of different kind. 

We have already mentioned that our goal is to prove Carleman estimates for such class of operators. To be precise, we will not only obtain Carleman estimates for the class of operators represented by $P_1$, but also for the adjoint class, that is the family of operators $P_1^*$ obtained by taking the adjoint of $P_1$, namely
\begin{align}
P_1^*&=\sum_{j=1}^n X_j^* X_jX_1^*+X_0^*,\label{eq.P1s}
\end{align}
which coincides with $P_1$ up to a differential operator of order two. 

The main motivation for the Carleman estimate for $P_1^*$, is our interest in the local solvability of $P_1$.
We recall that the role of Carleman estimates - which is well-known to be determinant in the analysis of unique continuation problems - is also crucial to prove local solvability properties of partial differential equations with variable coefficients. An estimate of this sort for an operator $P$, yields a local solvability result for the adjoint operator $P^*$. Hence, we can take advantage of the Carleman estimate for $P_1^*$ to get the local solvability of $P_1$.

The connection between our operator $P_1$ and some important operators coming from physics, can be seen by the analyzing the following two examples: the (linear) KdV operator and its higher dimensional generalization,  the ZK operator. 
The (linear) KdV operator is described by  $iP_1$ with $N=n=1$, $\R^{n+1}=\R^2_{t,x}$, $iX_0=\partial_t$. and $iX_1=-\partial_{x}$. However, by taking $iX_1(x)=a(x)\partial_x$, with $a(x)\neq0$ for all $x\in \mathbb{R}$,  then we get an operator, again of the form $P_1$, that describes a variable coefficient KdV operator.
If we consider $N=n=2$, $\R^{n+1}=\R_t\times\R^2_{x}$, $iX_0=\partial_t$ and $iX_j=-\partial_{x_j}$, $j=1,2$, then $iP_1$ describes the ZK operator. Again, by taking nonvanishing space-variable coefficient real vector fields satisfying condition $(\mathrm{H}_1)$ in \eqref{involutivity}, we get a variable coefficient ZK model. 

Note that both the classical KdV and the ZK operator (with constant coefficients) are built by taking an elliptic operator in $\mathbb{R}^n_x$, i.e. the Laplacian in $\R$ and $\R^2$ respectively, and then using the two vector fields $iX_0=\partial_t$ and $iX_1=\partial_{x_1}$ to define a third order operator. Observe that $P_1$ is built exactly the same way, but employing variable coefficient real vector fields instead of constant coefficient ones. In fact, to define $P_1$ we take operators being sums of squares, that is $\sum_{j=1}^N X_j^* X_j$ (not necessarily elliptic in a subspace of dimension $n$ of the $n+1$-dimensional domain), and cook up a third order operator by using $X_1$ and $X_0$. 
An example of variable coefficient KdV-type operator in a three-dimensional Euclidean space domain, can be constructed, for instance, by using the canonical basis of the Lie algebra of the Heisenberg group.
Let 
$$iX_0=\partial_t,\quad  iX_1(x,t)=\partial_{x_1}-\frac{x_2}{2}\partial_{x_3},\quad iX_2(x,t)=\partial_{x_2}+\frac{x_1}{2}\partial_{x_3},\quad iX_3(x)=\partial_{x_3},$$
be four real smooth vector fields, where $iX_1, iX_2$ are the generators of the stratified Lie algebra $\mathfrak{h}=\mathfrak{h}_1 \oplus \mathfrak{h}_2$ of $\mathbb{H}^1$.
 Then $\{iX_j\}_{j=1}^3$ is a global involutive distribution in $\R^4_{t,x}$,
 $$\mathcal{L}=\sum_{j=1}^3 X_j^* X_j$$ is an elliptic operator on $\mathbb{R}^3_x$, and
 $$X_{j_0}\sum_{j=1}^3 X_j^* X_j+X_0, \quad j_0=1,2,3,$$
 is a variable coefficient KdV-type operator of the form $P_1$ for every choice of $j_0=1,2,3$, since $X_{j_0}$ is nondegenerate for all $j_0=1,2,3$.

The examples above show that $P_1$-type operators include KdV and ZK variable coefficient PDOs. However, our class is not limited to be just a generalization of KdV-type operators in any space dimension, since we are free to choose $n\geq 3$, $N\leq n+1$, but also $iX_0\neq \partial_t$, giving rise to a very wide family of degenerate operators. 

Let us finally conclude this introduction with the plan of the paper. Section \ref{sec.LocSolvCarlemanEst} is devoted to the the main results of this work, specifically Theorem \ref{thm.Carleman} about a Carleman estimate for $P_1$ and for $P_1^*$,  and Theorem \ref{thm.solv} giving a local solvability  result for both the aforementioned operators. Section \ref{sec.example-applic} is dedicated to explicit examples of operators of the form under study and of KdV-type.
The last section, Section \ref{sec.finalRem-Applic}, contains final remarks and the discussion of possible applications in the context of dispersive equations.

\section{Carleman estimate and local solvability for $P_1$ and $P_1^*$}\label{sec.LocSolvCarlemanEst}
We recall once more that in this section we are concerned with the proofs of the main results of the article: a Carleman estimate for $P_1$ and $P_1^*$, and a local solvability result for these operators. The Carleman estimate is  the most important result of the paper, in that it is the key tool to get the solvability property. Moreover, it can find applications in other problems of our interest, like in uniqueness/unique continuation problems for dispersive equations. On the other hand, the investigation of the local solvability of highly degenerate operators is quite complicated, thus our results represent a step ahead in the understanding of certain kind of operators.

Before stating and proving our theorems, let us make precise the objects we are working with. 
Recall that, on $\mathbb{R}^{n+1}_{x}$, $x=(x_1,\ldots, x_n, x_{n+1})$, $n\geq 1$, we define the operator $P_1$ as
$$P_1=X_{1}\sum_{j=1}^N X_j^* X_j+X_0,\quad 1\leq N\leq n+1,$$
where 
\begin{align}
X_j(x,D_{x})&=\sum_{k=1}^{n+1}a_{jk}(x_1,\ldots,x_{n+1})D_k, \quad \forall j=1,\ldots, N \label{formula1.vs}\\
X_j^*(x,D_{x})&=\sum_{k=1}^{n+1}a_{jk}(x_1,\ldots,x_{n+1})D_k+\mathrm{d}_j(x_1,\ldots,x_{n+1})\\
&=X_j(x,D)+\mathrm{d_j}(x_1,\ldots,x_{n+1}),\quad \mathrm{d_j}(x):=\sum_{k=1}^{n+1}D_k a_{jk}(x_1,\ldots,x_{n+1}), \quad \forall j=1,\ldots, N\label{formula2.vs}
\end{align}
with $D_k=-i\partial_k:=-i\partial_{x_k}$, and $a_{jk}\in C^\infty(\R^{n+1})$, for all $j,k=1,\ldots, n+1$.

The condition we shall assume on the system of vector fields on $\R^{n+1}$ which make up $P_1$ is the following:
\begin{itemize}
\item [($\mathrm{H}_1$)]  We say that a system of real smooth vector fields $\{iX_l\}_{l=1,\ldots, N}$, $1\leq N\leq n+1$, satisfies condition ($\mathrm{H}_1$) if the distribution $\Gamma(\{iX_l\}_{l=1,\ldots, N})$ is a global involutive distribution, that is, for all $j,k=1,\ldots,N$,  there exist (uniquely determined) $c^{jk}_l\in C^\infty(\R^{n+1},i\R), l=1,\ldots,N$, such that
\begin{align}\label{involutivity}
    [X_j,X_k](x,D)=\sum_{l=1}^N c^{jk}_l(x) X_l(x,D), \quad \forall x\in\mathbb{R}^{n+1},
\end{align}
where $[X,Y]:=XY-YX$ denotes the commutator of the operators $X$ and $Y$.
\end{itemize}

In other words $(\mathrm{H}_1)$ amounts to the global involutivity of the system of real smooth vector fields $\{iX_j\}_{j=1}^N$ on $\R^{n+1} $.

\begin{remark}\label{remarkX1}
Due to the global involutivity of the system, and the fact that $N\geq 1$, there exists at least one vector field $iX_{j_0}\in \{iX_j\}_{j=1}^N$ being nondegenerate on $\R^{n+1}$, or, in other words, a nowhere vanishing vector field on $\R^{n+1}$. By possibly renaming the vector fields, we can always assume $j_0=1$, so we will use this convention in the rest of the paper when assuming condition $(\mathrm{H}_1)$ on a system of vector fields. This assumption will be crucial to derive the Carleman estimate below, and thus the solvability result as well.
Note also that, if $N\leq n+1$, and $\Gamma(\{iX_j\})$ is a global involutive distribution of $\mathrm{rank}=N$,  then $X_j$ is nondegenerate in $\R^{n+1}$ for all $j=1,\ldots,N$, hence we are free to take any vector field in the system as  the one playing the role of $X_1$.  
Finally, observe that the requirement $N\leq n+1$ is not restrictive. Due to the global involutivity, and the fact that the rank of the distribution is $M\leq n+1$, then if $N>n+1$ we can rewrite $P_1$ in terms of $M\leq n+1< N$ vector fields in the second order part.
\end{remark}

\begin{remark}\label{remarkH1-2}
 Notice that formula \eqref{involutivity} holds true independently of the rank of the distribution. If the system of vector fields satisfies $(\mathrm{H}_1)$, and the rank of $\Gamma(\{iX_l\}_{l=1,\ldots, N})=M\leq N\leq n +1$, then 
 \begin{align}
   [X_j,X_k](x,D)=\sum_{l=1}^M c^{jk}_l (x)X_l(x,D)=\sum_{l=1}^N c^{jk}_l(x) X_l(x,D),  
 \end{align}
 where some coefficients will be identically zero. Since we will not care about the precise rank of the distribution, we will simply use \eqref{involutivity}.
\end{remark}

\begin{remark}\label{remarkH1-2}
An additional remark is that, by imposing only ($\mathrm{H}_1$) on the system of vector fields, then $P_1$ generates a class of operators larger than the one of KdV-type operators, in that the latter is modeled by using a system of vector fields satisfying the additional conditions 
\begin{itemize}
\item $X_0(x,D)\neq 0, \quad \forall x\in\R^{n+1},$
    \item $[X_j,X_0]\equiv 0, \quad \forall j=1,\ldots,N$,
    \item $X_j(x,D)=X_j(x_1,\ldots,x_n, D_1,\ldots, D_n), \quad \forall j=1,\ldots,N,$
    \item $N=n$ and $\sum_{j=1}^n X_j^*X_j$ is elliptic on $\R^n_{(x_1,\ldots,x_n)}$.
\end{itemize}

\end{remark}

Since we will be giving a local solvability result for $P_1$ and $P_1^*$, for completeness we recall here the definition of $H^s-H^{s'}$ locally solvable partial differential operator at a point $x_0\in\Omega$, where, adopting the usual notation, $H^s(\Omega)$ stands for the Sobolev space of order $s$.

\begin{definition}
    Let $P$ be a partial differential operator with smooth coefficients defined on an open subset $\Omega$ of $\R^n$, and let $x_0\in\Omega$. We say that $P$ is $H^s-H^{s'}$ locally solvable at $x_0$ if there exists a compact $K$ containing $x_0$ in its interior $U_K$, such that, for every $f\in H^{-s}_{loc}(\Omega)$ there exists $u\in H^{-s'}_{loc}(\Omega)$ for which $Pu=f$ in $U_K$. 
We say that $P$ is $H^s-H^{s'}$ locally solvable in $\Omega$, if it is locally solvable at each point of $\Omega$.

\end{definition}

When $P$ is $H^s-H^{s'}$ locally solvable at $x_0$ with $s=s'=0$, we will just say that the operator is $L^2-L^2$ locally solvable at $x_0$.

\noindent\textbf{Notations.} Below $N$ will always be a positive integer, $\lambda$ will be a real number, whereas $A=\{a_{ji}(x)\}_{\substack{j=0,\ldots,N\\ i=1,\dots,n+1}}$ will denote the variable coefficient matrix containing the coefficients of the vector fields $iX_j$, for $j=0,\ldots,N$. 
 For a matrix $A$ as above and $\alpha\in \mathbb{N}^{n+1}\bigcup \{0\}^{n+1}$, we define
\begin{align}
    \|\nabla^\alpha A\|_{L^\infty(\Omega)}:=\sup_{\substack{j=1,\ldots,N\\ i=1,\dots,n+1}}\|\partial^\alpha a_{ji}\|_{L^\infty(\Omega)},\quad \forall \Omega\subset \mathbb{R}^{n+1},
\end{align}
where $\partial^\alpha f=\partial_1^{\alpha_1}\ldots \partial_{n+1}^{\alpha_{n+1}}$
Throughout the paper we will  write $\|\cdot\|_{L^\infty}$, $\|\cdot\|_{L^2}$, $\|\cdot\|_{H^s} $ without specifying the set $\Omega$ where the $L^\infty(\Omega)$, $L^2(\Omega)$, and the $H^s(\Omega)$ norms are taken. Given that we will be working with compactly supported functions, the set $\Omega$ will be the fixed support of the function (or its interior), so we omit it for simplicity. However, note that we can take $\|\nabla^\alpha A\|_{L^\infty}=\|\nabla^\alpha A\|_{L^\infty(\mathbb{R}^{n+1})}$ in the coming estimates when $A$ has $C^\infty_b(\mathbb{R}^{n+1})$ coefficients.
Finally, we shall use the notation $\Gamma(\{X_j\}_{j=1}^N)$ for the distribution spanned by the system of vector fields $\{X_j\}_{j=1}^N$.

In order to prove our Carleman inequality for $P_1$ and for $P_1^*$, we will make use of the subsequent fundamental lemma.

\begin{lemma}\label{lemma}
    Let $f\in C^\infty(\mathbb{R}^{n+1})$, $N\geq 1$, and $\lambda\geq 1$. Let also $P_1$ be as in \eqref{eq.P1} and $\{iX_j\}_{j=1}^N$ satisfying  $(\mathrm{H}_1)$.
    Then, for every compact set $K$ of $\mathbb{R}^{n+1}$, there exists a positive constant $C_K=C_K\left(\{\|\partial^\alpha f\|_{L^\infty(K)}\}_{|\alpha|=0,1,2},\right.$ $\left. \{\|\nabla^\alpha A\|_{L^\infty (K)}\}_{|\alpha|=0,1,2}\right)$ such that, for all $u\in C_0^\infty(K),$
\begin{align}
    &\im(e^{-\lambda f}P_1u, e^{-\lambda f}u)\\
    & \geq \small{\sum_{j=1}^N}(Q_j\ef,\ef)
    +\la^3 \sum_{j=1}^N((-iX_1f)|X_jf|^2\ef,\ef)-\lambda^2 C_K\|\ef\|_{L^2}^2,
\end{align}
where
\begin{align}
    Q_j=Q_j(x,D):=&\left(\la(-iX_1f)+id_1/2\right)X_j^*X_j+
   \sum_{k=1}^Nic^{j1}_k X^*_kX_j \\
   &-i\la (X_jX_1f) X_j+i/2\Big(c'_j-(X_1d_j)+2(X_jd_1)+d_1d_j\Big) X_j,\label{eqn.thm.Qj}
\end{align}
    with $c^{j1}_k,$ as in \eqref{involutivity}, and $ c'_j:=\sum_{k=1}^N(\sum_{l=1}^Nc^{lk}_j+(X_kc_j^{k1})+d_kc_j^{k1}-2d_kc_k^{j1})\in C^\infty (\R^{n+1},i\R)$ for all $j,k=1,\ldots, N$.
\end{lemma}

\begin{proof}

We start the proof by computing the quantity
    \begin{align}
        \mathrm{Im}(e^{-\la f}P_1u,e^{-\la f}u).\label{eqm.Car.001}
    \end{align}
Writing
\begin{align}
    e^{-\lambda f}P_1(x,D)u=e^{-\lambda f}P_1(x,D)e^{\lambda f}\ef= P^f_1(x,D)\ef,
\end{align}
then
\begin{align}
    P_1^f(x,D)&=P_1(x,D+\lambda Df)=
    (X_1+\la (X_1f))\sum_{j=1}^N(X_j+\la (X_jf))^*(X_j+\la (X_jf))+X_0+\la (X_0f)\\
    &=(X_1+\la (X_1f))\sum_{j=1}^N(X_j^*-\la (X_jf))(X_j+\la (X_jf))+X_0+\la (X_0f)\\
  & = (X_1+\la (X_1f))\sum_{j=1}^N\left[ X_j^*X_j+\la X_j^*(X_jf)-\la (X_jf)X_j-\lambda^2 (X_jf)^2\right]+X_0+\la (X_0f)\\
   & = (X_1+\la (X_1f))\sum_{j=1}^N\left[ X_j^*X_j+\la d_j(X_jf)+\la (X_j^2f)-\lambda^2 (X_jf)^2\right]+X_0+\la (X_0f)\\
   & = P_1 +X_1\sum_{j=1}^N\left[ \la d_j(X_jf)+\la (X_j^2f)-\lambda^2 (X_jf)^2\right]+\\
  &\quad  +\la (X_1f)\sum_{j=1}^N\left[ X_j^*X_j+\la d_j(X_jf)+\la (X_j^2f)-\lambda^2 (X_jf)^2\right]+\la (X_0f)\\
  & = P_1 +\sum_{j=1}^N\left[ \la d_j(X_jf)+\la (X_j^2f)-\lambda^2 (X_jf)^2\right]X_1+\\
  &\quad +\sum_{j=1}^N\left[ \la (X_1d_j)(X_jf)+\la d_j(X_1 X_jf)+\la (X_1X_j^2f)-2\lambda^2  (X_jf)(X_1X_jf)\right]\\
  & \quad +\la (X_1f)\sum_{j=1}^NX_j^*X_j+  (X_1f)\sum_{j=1}^N\left[\la^2 d_j(X_jf)+\la^2 (X_j^2f)+\lambda^3 (iX_jf)^2\right]+\la (X_0f)\\
  &=P_1+L_2+L_1+L_0,
\end{align}
where $L_2,L_1,L_0$ are the partial differential differential operators of order 2, 1, and 0, given,  respectively,  by
\begin{align}
    L_2&=L_2(x,D):=\la (X_1f)\sum_{j=1}^NX_j^*X_j\\
    L_1&=L_1(x,D):=\sum_{j=1}^N\left[ \la d_j(X_jf)+\la (X_j^2f)-\lambda^2 (X_jf)^2\right]X_1\\
    L_0&=L_0(x):=
    \sum_{j=1}^N\left[\la (X_1d_j)(X_jf)+\la d_j(X_1 X_jf)+\la (X_1X_j^2f)-2\lambda^2  (X_jf)(X_1X_jf)+\right.\\
    &\left.\quad +(X_1f)(\la^2 d_j(X_jf)+\la^2 (X_j^2f)+\lambda^3 (iX_jf)^2)\right]+\la (X_0f)
\end{align}
By the calculations above 
\begin{align}
        \mathrm{Im}(e^{-\la f}P_1u,e^{-\la f}u)&=\mathrm{Im}(P_1\ef,e^{-\la f}u)+ \mathrm{Im}(L_2\ef,e^{-\la f}u)+\mathrm{Im}(L_1\ef,e^{-\la f}u)+\mathrm{Im}(L_0\ef,e^{-\la f}u),
    \end{align}
so, for simplicity, we handle each term on the right-hand side separately.

Recall that, since the system of vector fields $\{iX_j\}_{j=1,\ldots,N}$ is globally involutive, that is satisfies $(\mathrm{H}_1)$, for all $j=1,\ldots, N$, there exist $c^{j1}_k\in C^\infty(\R^{n+1},i\R), k=1,\ldots,N$,  such that
\begin{align}
    [X_j,X_1]=\sum_{k=1}^N c^{j1}_kX_k.
\end{align}
By the same token, there exist  some $c^{jk}_l\in C^\infty(\R^{n+1},i\R), l=1,\ldots,N$, such that
\begin{align}
    [X_j,[X_j,X_1]]=\sum_{k,l=1}^N c^{j1}_kc_l^{jk}X_l.
\end{align}
Now, on putting $v:=\ef$, and using the previous observation, we have
\begin{align}
    2i\mathrm{Im}(P_1\ef,\ef)&=2i\sum_{j=1}^N \mathrm{Im}(X_1X_j^*X_jv,v)+2i\,\mathrm{Im}(X_0v,v)\\
    &=\sum_{j=1}^N \left( (X_1X_j^*X_j-X_j^*X_jX_1-X_j^*X_jd_1)v,v\right)+\left((X_0-X_0^*)v,v\right)\\
    &=-\sum_{j=1}^N \left( (2[X_j,X_1]X_j+ [X_j,[X_j,X_1]]+ d_j[X_j,X_1]-(X_1d_j)X_j)v,v\right)\\
    &-\sum_{j=1}^N\Bigg( \Big( d_1X_j^*X_j+2(X_jd_1)X_j+ d_1d_jX_j+ (X_j^*X_jd_1) +(X_jd_1)d_j\Big)v,v\Bigg)-(d_0v,v)\\
    &=-\sum_{j,k=1}^N (2c^{j1}_k X_kX_j v,v)
    -d_1\sum_{j=1}^NX_j^*X_j
    -\sum_{j,k,l=1}^N (c^{jk}_l X_lv,v)-\sum_{j,k=1}^N\left((X_jc_k^{j1})X_kv,v\right)\\
    &-\sum_{j,k=1}^N (d_jc_k^{j1} X_kv,v)
-\sum_{j=1}^n\Bigg( \Big( -(X_1d_j)+2(X_jd_1)+d_1d_j  \Big)X_j v,v\Bigg)\\
&-\sum_{j=1}^N\Bigg( \big( (X_j^*X_jd_1) +(X_jd_1)d_j\big)\,v,v\Bigg)-(d_0v,v).\\
\end{align}
Adopting the notation $\|\cdot\|_{L^\infty}:=\|\cdot \|_{L^\infty(K)}$, we get
\begin{align}
    \mathrm{Im}(P_1 v,v)
    &=\sum_{j,k=1}^N (ic^{j1}_k X_kX_j v,v)+ \frac{i}{2}d_1\sum_{j=1}^N(X_j^*X_jv,v)\\
    &+ \frac{i}{2}\sum_{j=1}^N(\sum_{k=1}^N(\sum_{l=1}^Nc^{lk}_j+(X_kc_j^{k1})+d_kc_j^{k1}) X_jv,v)
+ \frac{i}{2}\sum_{j=1}^n\left( \Big( -(X_1d_j)+2(X_jd_1)+d_1d_j  \Big)X_j v,v\right)\\
&+ \frac{i}{2}\sum_{j=1}^N\left( ( (X_j^*X_jd_1) +(X_jd_1)d_j)\,v,v\right)+ \frac{i}{2}(d_0v,v) \label{estP1}\\
&\geq \sum_{j=1}^N \Bigg( \Big(\sum_{k=1}^Nic^{j1}_k X_k^*X_j+\frac i 2 d_1X_j^*X_j+i/2\Big(c'_j-(X_1d_j)+2(X_jd_1)+d_1d_j\Big) X_j\Big)v,v\Bigg)-\!C_K\|v\|_{L^2}^2,\end{align}
where $c'_j(x):=\sum_{k=1}^N(\sum_{l=1}^Nc^{lk}_j+(X_kc_j^{k1})+d_kc_j^{k1}-2d_kc_k^{j1})(x) $,  while $C_K$ denotes a constant depending on the $L^\infty(K)$-norm of some smooth functions, specifically functions depending on the coefficients of the vector fields $iX_j$ and on their derivatives. Abusing of the notation, below we shall write $C_K$ for any constant of this form, that is any constant depending on the $L^\infty(K)$-norms of $f$, of the coefficients of the vector fields $iX_j,j=0,\ldots,N$, and of the derivatives of these functions. Sometimes $C_K$ will also depend on the $L^\infty(K)$-norm of $c_j^{kl}$, $k,l=1,\ldots, N$, $j=\,\ldots,n+1$.

As for  the term containing $L_2$, recalling that $f$ is a real function and that $X_1f$ takes purely imaginary values, we obtain
\begin{align}
    \im(L_2v,v)&=\im(\la (X_1f)\sum_{j=1}^NX_j^*X_j v,v)\\
    &\frac{-i}{2}\la \sum_{j=1}^N( ((X_1f)X_j^*X_j-X_j^*X_j(-X_1f)) v,v)\\
    &=\la\sum_{j=1}^N\left( \left((-iX_1f)X_j^*X_j v,v\right)-i \left((X_jX_1f) X_j v,v\right)-\frac i 2\left (( (X_j^2X_1f)+(X_jX_1f)+d_j(X_jX_1f)) v,v\right ) \right)\\
    &\geq \la\sum_{j=1}^N\left((-iX_1f)X_j^*X_j v,v)-i ((X_jX_1f) X_j v,v\right)-C_K \|v\|_{L^2}^2\label{estL2}.
\end{align}

By similar considerations 
\begin{align}
    \im(L_1v,v)=&\frac{1}{2i}\sum_{j=1}^n\big\{
   \la \left((d_1d_j(X_jf)-d_1(X_j^2f)+(X_1d_j(X_jf)) -(X_1X_j^2f))v,v \right) 
   \\ &+\la^2 \left(( d_1(X_jf)^2+2(X_jf)(X_1X_jf) )v,v\right)\big\}\\
   \geq& -\la^2 C_K\|v\|_{L^2}^2,\label{estL1}
\end{align}
and
\begin{align}
    \im(L_0v,v)\geq& \la^3\sum_{j=1}^N((-iX_1f)|X_jf|^2v,v)+\lambda (-i(X_0f)v,v)-\lambda^2 C_K\|v\|_{L^2}^2\\
    \geq& \la^3 \sum_{j=1}^N((-iX_1f)|X_jf|^2v,v)-\lambda^2 C_K\|v\|_{L^2}^2.\label{estL0}
\end{align}

Finally, by \eqref{estP1}, \eqref{estL2}, \eqref{estL1}, \eqref{estL0}, and the fact that $\la\geq 1$, we can find a new suitable positive constant $C_K$ such that
\begin{align}
    \im(e^{-\lambda f}P_1u,\ef)\geq \sum_{j=1}^N (Q_j v,v)+\la^3 \sum_{j=1}^N((-iX_1f)|X_jf|^2v,v)-\lambda^2 C_K\|v\|_{L^2}^2,\quad \forall u\in C_0^\infty(K),
\end{align}
with
\begin{align}
    Q_j=Q_j(x,D):=&(\la(-iX_1f)+id_1/2)X_j^*X_j+
   \sum_{k=1}^Nic^{j1}_k X_k^*X_j \\
   &-i\la (X_jX_1f) X_j+i/2\Big(c'_j-(X_1d_j)+2(X_jd_1)+d_1d_j\Big) X_j.
\end{align}
This completes the proof.
\end{proof}

The same result as Lemma \ref{lemma}, with suitable small adjustments, is still valid for the operator $P_1^*$, the adjoint of $P_1$. We state the precise result in Lemma \ref{lemma.P1adjoint} below.

\begin{lemma}\label{lemma.P1adjoint}
 Let $f\in C^\infty(\mathbb{R}^{n+1})$, $N\geq 1$, and $\lambda\geq 1$. Let also $P_1^*$ be as in \eqref{eq.P1s}, and $\{iX_j\}_{j=1}^N$ satisfying  $(\mathrm{H}_1)$.
    Then, for every compact set $K$ of $\mathbb{R}^{n+1}$, there exists a positive constant  $C_K=C_K\left(\{\|\partial^\alpha f\|_{L^\infty(K)}\}_{|\alpha|=0,1,2},\right.$ $\left. \{\|\nabla^\alpha A\|_{L^\infty (K)}\}_{|\alpha|=0,1,2}\right)$ such that, for all $u\in C_0^\infty(K),$
\begin{align}
    &\im(e^{-\lambda f}P_1^*u, e^{-\lambda f}u)\\
    & \geq \small{\sum_{j=1}^N}(Q'_j\ef,\ef)
    +\la^3 \sum_{j=1}^N((-iX_1f)|X_jf|^2\ef,\ef)-\lambda^2 C_K\|\ef\|_{L^2}^2
\end{align}
where
\begin{align}
       Q'_j(x,D)=&Q_j(x,D)-i\sum_{k=1}^N c_k^{j1}(X_j^*X_k+X_k^*X_j)-ic''_jX_j\\
        =&(\la(-iX_1f)+id_1/2)X_j^*X_j-
   \sum_{k=1}^Nic^{j1}_k X_j^*X_k \\
   &-i\la (X_jX_1f) X_j+i/2\Big(c'_j-c''_j-(X_1d_j)+2(X_jd_1)+d_1d_j\Big) X_j\label{eqn.thm.Qjp}
    \end{align}
 with $c^{j1}_k,$ as in \eqref{involutivity}, $ c''_j:=\sum_{k=1}^N\Big( (X_kc_k^{j1})+(X_k^*c_k^{j1}) +(X_kc_j^{k1})+(X_k^*c_j^{k1})\Big)\in C^\infty (\R^{n+1},i\R)$ for all $j,k=1,\ldots, N$, and $Q_j$, $c_j'$, as in \eqref{eqn.thm.Qj} for all $j=1,\ldots,N$.
\end{lemma}

\begin{proof}
To prove this lemma we make use of Lemma \ref{lemma}. By standard arguments we have 
\begin{align}
    P_1^*(x,D)=&P_1(X,D)+d_1\sum_{j=1}^nX_j^*X_j+\sum_{j,k=1}^N\left\{c^{j1}_{k}(X_j^*X_k+X_k^*X_j) -c^{j1}_{k}d_kX_j+(X_k^*c_j^{k1})X_j)\right\} \\
&-\sum_{j=1}^N(X_1d_j)X_j+\sum_{j=1}^N((X_jd_1)d_j+(X_j^*X_1d_1))+d_0,
\end{align}
where $c^{j1}_{k}\in C^\infty(\R^{n+1},i\R)$ are such that 
\begin{align}
    [X_j,X_1]=\sum_{k=1}^N c^{j1}_{k}X_k.
\end{align}
Next, as in the proof of Lemma \ref{lemma}, we compute
\begin{align}
    \im(e^{-\la f}P_1^*u,\ef)&= \im({P_1^*}^f\ef ,\ef),
\end{align}
where
\begin{align}
    {P_1^*}^f =&e^{-\la f}P_1^* e^{\la f}=P_1^*(x,D+\la Df)\\
    =&P_1^f+d_1\sum_{j=1}^nX_j^*X_j+\sum_{j,k=1}^N\left\{c^{j1}_{k}(X_j^*X_k+X_k^*X_j) +(-c^{j1}_{k}d_k+(X_k^*c_j^{k1}))X_j\right\} -\sum_{j=1}^N(X_1d_j)X_j\\
    &+\sum_{j=1}^N((X_jd_1)d_j+(X_j^*X_1d_1))+d_0+d_1\sum_{j=1}^N\left[ \la d_j(X_jf)+\la (X_j^2f)-\lambda^2 (X_jf)^2\right]\\
    &+\sum_{j,k=1}^Nc_k^{j1}\left\{  \la (X_kf)(X_j^*-X_j)+\la(X_jf)(X_k^*-X_k)+\la (X_j^*X_kf) +\la (X_k^*X_jf)-2\la^2(X_jf)(X_kf)\right\}\\
    &-\la \sum_{j=1}^N\left\{\sum_{k=1}^N(c^{j1}_kd_k-(X_k^*c_j^{k1}))(X_jf)+(X_1d_j)(X_jf)\right\}.\label{eq.P1stf}
\end{align}
    Using the same notations as in the proof of Lemma \ref{lemma}, we have
    \begin{align}
    \im({P_1^*}^fv ,v)\geq &\im(P_1^f v,v)
    +\sum_{j,k=1}^N\im\left(\left(c^{j1}_{k}(X_j^*X_k+X_k^*X_j) -(c^{j1}_{k}d_k-(X_k^*c_j^{k1}))X_j\right)v,v\right)\\
    &-\sum_{j=1}^N\im((X_1d_j)X_jv,v)
    -\la^2 C_K \|v\|_{L^2}^2 \\
    \geq & \sum_{j=1}^N(Q_jv,v)+\la^3 \|\, |X_1f|^{\frac 3 2}v\|_{L^2}^2+\sum_{j,k=1}^N\im\left(c^{j1}_{k}(X_j^*X_k+X_k^*X_j) v,v\right)\\
    &-\sum_{j=1}^N \im\left(\sum_{k=1}^N(c^{j1}_{k}d_k-(X_k^*c_{j}^{k1}))X_jv,v\right)
    -\sum_{j=1}^N\im((X_1d_j)X_jv,v)-\lambda^2 C_K\|v\|_{L^2}^2,\label{pf.eq.1}
    \end{align}
    where $Q_j$ is as in \eqref{eqn.thm.Qj}.

Next, we define $g_j(x):=\sum_{k=1}^N(c^{j1}_{k}d_k-(X_k^*c_{j}^{k1}))$, and have that, since $g_j$ and $(X_1d_j)$ are real valued functions, and since the $c_k^{j1}$ take purely imaginary values, then
\begin{align}\label{eqn27121603}
    \im\left(g_j X_jv,v\right)=-\frac{1}{2i}\left(((X_jg_j)+(d_jg_j))v,v\right)\geq -C_K\|v\|^2_{L^2},
\end{align}
\begin{align}\label{eqn27121602}
    \im((X_1d_j)X_jv,v)=-\frac{1}{2i}\left(((X_jX_1d_j)+(X_1d_j)d_j)v,v \right)\geq -C_K\|v\|^2_{L^2},
\end{align}
and
\begin{align}
    \im\left(c^{j1}_{k}(X_j^*X_k+X_k^*X_j) v,v\right)=& (-ic_k^{j1}(X_j^*X_k+X_k^*X_j)v,v)\\
    &-\frac i 2 \left( \Bigg( \Big ((X_jc_k^{j1})+(X_j^*c_k^{j1})\Big)X_k+ \Big ((X_kc_k^{j1})+(X_k^*c_k^{j1})\Big)X_j\Bigg) v,v\right)\\
    &-\frac i 2\left(\Big((X_jc_k^{j1})d_k +(X_kc_k^{j1})d_j +(X^*_kX_jc_k^{j1}+ X_j^*X_kc_k^{j1})\Big)v,v\right),
\end{align}
which gives
\begin{align}
  \sum_{j,k=1}^N   \im\left(c^{j1}_{k}(X_j^*X_k+X_k^*X_j) v,v\right)\geq & \sum_{j,k=1}^N(-ic_k^{j1}(X_j^*X_k+X_k^*X_j)v,v)-\frac i 2\sum_{j=1}^N (c''_j X_j v,v)-C_K\|v\|_{L^2}^2\label{eqn27121601}
\end{align}
where
\begin{align}
    c''_j:=\sum_{k=1}^N\Big( (X_kc_k^{j1})+(X_k^*c_k^{j1}) +(X_kc_j^{k1})+(X_k^*c_j^{k1})\Big).
\end{align}
    Inserting the last three inequalities, that is  \eqref{eqn27121601}, \eqref{eqn27121602} and  \eqref{eqn27121603}, into \eqref{pf.eq.1},  and using that $\lambda\geq 1$, by rearranging the indices we conclude
    \begin{align}
     \im({P_1^*}^fv ,v)\geq   \sum_{j=1}^N(Q'_jv,v)+\la^3 \|\, |X_1f|^{\frac 3 2}v\|_{L^2}^2-\lambda^2 C_K\|v\|_{L^2}^2,
    \end{align}
    with
    \begin{align}
        Q'_j(x,D)=&Q_j(x,D)-i\sum_{k=1}^N c_k^{j1}(X_j^*X_k+X_k^*X_j)-ic''_jX_j\\
        =&(\la(-iX_1f)+id_1/2)X_j^*X_j-
   \sum_{k=1}^Nic^{j1}_k X_j^*X_k \\
   &-i\la (X_jX_1f) X_j+i/2\Big(c'_j-c''_j-(X_1d_j)+2(X_jd_1)+d_1d_j\Big) X_j, \label{eqn29121}\\
    \end{align}
    that is $Q_j'$ differs from $Q_j$ in the coefficients of the second term in \eqref{eqn29121}, and in the appearance of the function $c'_j-c''_j$ instead of $c'_j$ in the last term in \eqref{eqn29121}. This completes the proof.
\end{proof}
\begin{remark}
    The estimates proved so far do not necessitate of the nondegeneracy of $X_1$. The  global nonvanishing requirement on $X_1$ will come into play in the subsequent results.
\end{remark}
In the next theorem we will refer to the so called H\"ormander's condition of rank $r$ on a system of \textit{real} smooth vector fields. Let us recall that this means that the system of vector fields, along with their commutators of length $r$ (being assumed that $X_1$, $[X_1,X_2]$, $[X_1,[X_2,X_3]]$,...  are counted as commutators of length 1, 2, 3,.. and so on), generate the whole tangent space at each point, namely, in our case, $\mathbb{R}^{n+1}$. 

\begin{theorem}[Carleman estimate for $P_1$ and $P_1^*$]\label{thm.Carleman}
    Let $f\in C^\infty(\mathbb{R}^{n+1})$,
    $\{iX_i\}_{i=1}^N$ be a system of real smooth vector fields satisfying  $(\mathrm{H}_1)$, and $K$ a compact subset of $\R^{n+1}$.
    Then,
    \begin{itemize}
    \item[(i)] If $f\in C^\infty(\R^n)$ is such that there exists $C_0>0$ for which
    \begin{align}
        -iX_1f(x)\geq C_0>0, \quad \forall x\in K,
    \end{align}
     then there exist $\lambda_0=\lambda_0(K)\geq 1$ and $C=C(K)>0$ such that, for every $\lambda\geq\lambda_0$, 
   \begin{equation}
        \|e^{-\lambda f}P_1^*u\|_{L^2}^2,\|e^{-\lambda f}P_1u\|_{L^2}^2\geq C \lambda \|e^{-\lambda f}u\|_{L^2}^2, \quad \forall u\in C_0^\infty(K).
    \end{equation}
        \item [(ii)] If the system of vector fields $\{iX_i\}_{i=1}^n$ satisfies H\"ormander's condition of rank $r$, and if there exists $C_0>0$ for which 
        $$-iX_1f(x)\geq C_0, \quad \forall x\in K,$$
        then there exist $\lambda_0=\lambda_0(K)\geq 1$ and $C=C(K)>0$ such that, for every $\lambda\geq \lambda_0$, 
      \begin{equation}
       \|e^{-\lambda f}P_1^*u\|_{L^2}^2, \|e^{-\lambda f}P_1u\|_{L^2}^2\geq C \la \|e^{-\lambda f}u\|_{H^{\frac{1}{r}}}^2,\quad \forall u\in C_0^\infty(K).
    \end{equation}
    \end{itemize}
\end{theorem}

\begin{proof}
The proof relies on the use of Lemma \ref{lemma} and of Lemma \ref{lemma.P1adjoint} when dealing with the statement for $P_1$ and for $P_1^*$ respectively.
Below we shall give a detailed proof of the result for $P_1$, and we omit the one for $P_1^*$. This is done because the proofs of the estimates for $P_1^*$ and for $P_1$ are exactly the same, up to the appearance of different constants depending on the $L^\infty(K)$-norm of some smooth functions (because the difference between the estimates in Lemma \ref{lemma} and Lemma \ref{lemma.P1adjoint} is just in the appearance of possibly different smooth coefficients in $Q_j'$ and $Q_j$).  Since there is no substantial change in the two cases, we focus on $P_1$.

By Lemma \ref{lemma}, for all $\lambda\geq 1$, given a compact $K$ of $\mathbb{R}^n$ on which $-iX_1f\geq C_0>0$, we have
\begin{align}
    &\im(e^{-\lambda f}P_1u, e^{-\lambda f}u)\\
    & \geq \small{\sum_{j=1}^N}(Q_j \ef,\ef)
    +\la^3 \|\, |X_1f|^{\frac 3 2}\ef\|_{L^2}^2-\lambda^2 C_K\|\ef\|_{L^2}^2, \quad \forall u\in C_0^\infty (K),
\end{align}
with 
\begin{align}
    Q_j=Q_j(x,D):=&(\la(-iX_1f)+id_1/2)X_j^*X_j+
   \sum_{k=1}^Nic^{j1}_k X_k^*X_j \\
   &+\Big(-i\la (X_jX_1f) +i/2(c'_j-(X_1d_j)+2(X_jd_1)+d_1d_j)\Big) X_j,
\end{align}
   where $i c^{j1}_k, i c'_k$, $k=1,\ldots, N$, are smooth and real valued functions.
Therefore, once $\lambda$ is big enough, the proof of $(i)$ and of $(ii)$ depends on the estimate we can prove for the second order term $\small{\sum_{j=1}^n}\re(Q_j\ef,\ef)$. Then, let us focus on this term and show how different hypotheses lead to different lower bounds.

In the sequel we will use again the notation $v:=\ef$. First observe that, by Cauchy-Schwarz inequality,
\begin{align}
    \lambda \small{\sum_{j=1}^N}\re((-i X_1f)X_j^*X_j v,v)&\geq \lambda(C_0-\frac{\delta_0}{2})\small{\sum_{j=1}^N} \|X_jv\|_{L^2}^2-\frac{\lambda N}{2\delta_0}\max_{j=1,\ldots,N}\|X_jX_1f\|_{L^\infty} \|v\|_{L^2}^2\\
    &\geq \lambda(C_0-\frac{\delta_0}{2})\small{\sum_{j=1}^N} \|X_jv\|_{L^2}^2-\lambda c_0\|v\|_{L^2}^2, \forall \delta_0\in(0,1],
\end{align}
with $c_0=c_0(\delta_0,\{\|\partial^\alpha f\|_{L^\infty}\}_{|\alpha|=0,1,2}, \{\|\nabla^\alpha A\|_{L^\infty}\}_{|\alpha|=0,1,2})$,
and
\begin{align}
    &\frac i 2\sum_{j=1}^N(d_1X_j^*X_jv,v)-\sum_{j,k=1}^N ic^{j1}_k(X_k^* X_j v,v)\\
    &=\frac i 2\sum_{j=1}^N\Big((d_1X_jv,X_jv)-(X_jv,(X_jd_1)v) \Big)-\sum_{j,k=1}^N (ic^{j1}_k X_j v,X_k  v)-\sum_{j}^N ( X_j v,\sum_{k=1}^N(X_k ic^{j1}_k) v)\\
    &\geq -( \|d_1\|_{L^\infty}+C_1 N+\frac{\delta_0'}{2})\sum_{j=1}^N \|X_jv\|_{L^2}^2-\frac{C_K}{\delta_0'}\|v\|_{L^2}^2,
\end{align}
with a new positive constant $C_1:=\max_{j,k=1,\ldots,N}\{ |c_k^{j1}|\}$, and where $C_K$ is a positive constant depending on all the $L^\infty$-norms of the functions $X_jd_1$ and $X_kc_k^{j1}$.
By similar considerations,
  $$ \Big( \Big(-i\la (X_jX_1f) +i/2(c'_j-(X_1d_j)+2(X_jd_1)+d_1d_j)\Big) X_j v,v\Big)\geq \frac{\delta_1}{2}\sum_{j=1}^N\|X_jv\|^2_{L^2}-\la ^2 \frac{C_K}{\delta_1}\|v\|^2_{L^2},\quad \forall\delta_1\in (0,1],$$
where $C_K$ is a different suitable positive constant depending on the $L^\infty$-norms of $f,a_{jk}$, $c'_j$ and of their derivatives. Therefore,
\begin{align}
    \sum_{j=1}^N(Q_jv,v)\geq \la( C_0-\frac{\delta_0}{2}-\frac{\delta_0'}{2}-\frac{1}{\la}\|d_1\|_{L^\infty}-\frac{1}{\la}C_1N-\frac{\delta_1}{2})\|X_j v\|_{L^2}^2-\la ^2 c_1\|v\|^2_{L^2},
\end{align}
with $c_1=c_1(\delta_0,\delta_0',\delta_1,\{\|\partial^\alpha f\|_{L^\infty}\}_{|\alpha|=0,1,2}, \{\|\nabla^\alpha A\|_{L^\infty}\}_{|\alpha|=0,1,2})$. Then, choosing $\lambda>4(C_1N+\|d_1\|_{L^\infty})C_0^{-1}$ big enough, and $\delta_0,\delta_0',\delta_1$ sufficiently small, so that
$$C_0-\frac{\delta_0+\delta_0'}{2}-\frac{1}{\la}(C_1N+\|d_1\|_{L^\infty})-\frac{\delta_1}{2}\geq \frac{C_0}{2},$$
we get
\begin{align}
  \sum_{j=1}^N(Q_jv,v)\geq \la\frac{C_0}{2}\sum_{j=1}^N\|X_j v\|_{L^2}^2-\la ^2 C_K\|v\|^2_{L^2},  
\end{align}
where $C_K$ is a new positive constant since we have fixed $\delta_0,\delta_0',\delta_1$.

The estimate above yields
\begin{align}
     &\im(e^{-\lambda f}P_1u, e^{-\lambda f}u) \geq \la\frac{C_0}{2}\small{\sum_{j=1}^N}\|X_j v\|_{L^2}^2
    +\la^3 \|\, |X_1f|^{\frac 3 2}v\|_{L^2}^2-\lambda^2 C_K\|v\|_{L^2}^2,
\end{align}
hence, if $\lambda>\max\{4(C_1N+\|d_1\|_{L^\infty})C_0^{-1},2 C_K/C_0^3\}=:\la_0=\la_0(K)$, we have
\begin{align}
     &\im(e^{-\lambda f}P_1u, e^{-\lambda f}u)\geq \la\frac{C_0}{2}\small{\sum_{j=1}^N}\|X_j v\|_{L^2}^2
    +\la^2 \frac{C_K}{2}\|v\|_{L^2}^2,
\end{align}
where, recall, $C_K$ is a constant depending on some fixed constants $N, \delta_0,\delta_0',\delta_1$ and on the $L^\infty$-norms on $K$ of $f, a_{jk}$, $c'_j$, and of their derivatives.

Finally, by Cauchy-Schwarz inequality,
\begin{align}
    \|e^{-\lambda f}P_1 u\|_{L^2}^2\geq 2\delta\Big(\la\frac{C_0}{2}\small{\sum_{j=1}^N}\|X_j v\|_{L^2}^2
    +(\la^2 \frac{C_K}{2}-\frac{\delta}{2}) \|v\|_{L^2}^2\Big),
\end{align}
which gives, taking $\delta$ small enough and not necessarily depending on $K$ (since we can always assume $\frac{C_K}{2}\cdot\lambda>1$ and take, for instance, $\delta=1/2<\la^2\frac{C_K}{4}$), we reach
\begin{align}
    \|e^{-\lambda f}P_1 u\|_{L^2}^2\geq \la\frac{C_0}{2}\small{\sum_{j=1}^N}\|X_j v\|_{L^2}^2
    +\la^2 \frac{C_K}{4}\|v\|_{L^2}^2,\label{thm.carlem.final.est}
\end{align}
which proves part (i) of the theorem.

To prove part (ii) of the statement, we apply the Rothschield-Stein inequality for systems of vector fields satisfying H\"ormander's condition of rank $r$ on the first term on the right hand side of \eqref{thm.carlem.final.est}, that is
$$ \sum_{j=1}^N\|X_j v\|_{L^2}^2 \geq C \|v\|_{H^{\frac{1}{r}}},\quad \forall v\in C_0^\infty (K).$$
This completes the proof.
\end{proof}

\begin{remark}[The case $N=n=1$]
    Note that when $N=n=1$ $(\mathrm{H}_1)$ is trivially satisfied. For the same reason, result (ii) does not apply to this case, since $X_1$ alone cannot generate $\mathbb{R}^2$.
\end{remark}

An immediate application of the Carleman estimates in Theorem \ref{thm.Carleman} is the following local solvability result.

\begin{theorem}\label{thm.solv}
 Let $\{iX_j\}_{j=1}^N$ be a system of vector fields satisfying   $(\mathrm{H}_1)$ such that $X_1(x,D)\neq 0$ for all $x\in \R^{n+1}$. Let also $P_1$ and $P_1^*$ be as in \eqref{eq.P1} and \eqref{eq.P1s} respectively. Then the following properties hold.
\begin{itemize}
    \item [(i)] $P_1$ and $P_1^*$ are $L^2-L^2$ locally solvable at any point of $\mathbb{R}^{n+1}$.
\item[(ii)] If $\{iX_j\}_{j=1}^N$ satisfies H\"ormander's condition of rank $r$, $P_1$ and $P_1^*$ are $H^{-1/r}-L^2$ locally solvable at any point of $\mathbb{R}^{n+1}$.
\end{itemize}
\end{theorem}

\begin{proof}
    The proof is based on the use of Theorem \ref{thm.Carleman} and of some standard functional analysis argument.
    More specifically, the $H^s-L^2$ local solvability of $P_1$ at a point $x\in\mathbb{R}^{n+1}$, is equivalent to the validity of the following \textit{solvability estimate} for the adjoint $P_1^*$: there exists $C>0$ and a compact $K$ containing $x$ in its interior $U_K$, such that
    $$\|P_1^*u\|_{L^2}\geq C \|u\|_{H^{-s}}, \quad \forall u\in C_0^\infty(U_K).$$
    The solvability estimate, Hahn-Banach and Riesz representation theorems, alltogether give the result (see, for instance, Lemma 1.2.30 in \cite{L2}).
    Since $X_1$ is nondegenerate (see also Remark \ref{remarkX1}), given $C_0>0$ and $x_0\in\R^{n+1}$, it is always possible to find a smooth function $f$ and a compact set $K\subset \mathbb{R}^{n+1}$ containing $x_0$ in its interior, such that $iX_1f(x)\geq C_0>0$ for all $x\in K$.
    Hence, by the Carleman estimate for $P_1^*$, for $x_0\in\mathbb{R}^{n+1}$, and for the compact $K$ chosen as above and containing $x_0$ in its interior, there exists $\lambda_0(K)>1$  and $C>0$ such that, for all $\lambda>\lambda_0$,
    \begin{align}
      \|e^{-\lambda f}P_1^*\|_{L^2}\geq C  \|\ef\|_{H^{-s}}, \quad \forall u\in C_0^\infty(K), \label{eq.thm.solv.1}
    \end{align}
    where $s=0$ in case (i) and $s=-1/r$ in case (ii).
By estimating from above and from below $e^{-\lambda f}$ on $K$ with positive constants, \eqref{eq.thm.solv.1} easily implies the solvability estimate with $s=0$ under the hypothesis in part (i), and with $s=-1/r$ under the hypothesis in part (ii). From the solvability estimate, by the standard considerations mentioned at the beginning of the proof (see \cite{L2}) we get the $L^2-L^2$ and $H^{-\frac{1}{r}}-L^2$ local solvability of $P_1$ at $x_0$ according to the hypotheses in (i) and (ii) respectively. Finally, by the arbitrariness of $x_0$, we conclude the local solvability of $P_1$ at any point of $\R^{n+1}$.

For the local solvability of $P_1^*$, one repeats the same steps reversing the roles of $P_1$ and $P_1^*$.
\end{proof}

\begin{remark}
    When the vector field $iX_1$ is nondegenerate only on a compact $K \subset \R^{n+1}$, using the previous strategy we can reach the local solvability of $P_1$ and of $P_1^*$ at any point of the interior  $U_K$ of $K$.
\end{remark}

\section{Examples and applications: KdV and non KdV-type operators with variable coefficients}\label{sec.example-applic}
In this section we would like to list a few examples of operators of $P_1$-type. We pointed out more than once that we are particularly interested in dispersive models, therefore we will mainly focus on such cases here, even though the class is more general and we could exhibit different type of operators.

\noindent\textbf{Variable coefficient KdV operators.}
Let $N=n=1$, $\mathbb{R}^{n+1}=\mathbb{R}^2=\mathbb{R}_t\times \mathbb{R}_x$, and
\begin{align}
    X_0(t,x,D_t,D_x)=D_t,\quad \text{and}\quad X_1(t,x,D_t,D_x)=a(x)D_x,
\end{align}
where $a\in C^\infty(\mathbb{R}_x;\R)$, $a(x)\neq 0$ for all $x\in \mathbb{R}$.
Then 
\begin{align}
    P_1(t,x,D_t,D_x)&=D_t+ (a(x)D_x)(a(x)D_x)^*(a(x)D_x)\\
    &=-i(\partial_t+(a(x)\partial_x)(a(x)\partial_x)^*(a(x)\partial_x))
\end{align}
which is a KdV operator with space-variable coefficients. Our solvability theorem applies to this model, so we can assert that this operator, as well as $iP_1$ of course, is $L^2-L^2$ locally solvable at any point of $\mathbb{R}^2$.
 More generally, one can consider the coefficient $a$ as a function of both $(t,x)$, hence $a(t,x)$, or even just a time dependent function, provided that tit is everywhere nonvanishing.
In this last case one can also describe some of the operators treated in \cite{AJAC}.

\noindent \textbf{Variable coefficient ZK operators.}
Let $N=n=2$, $\mathbb{R}^{n+1}=\mathbb{R}^3=\mathbb{R}_t\times \mathbb{R}^2_x$, and
\begin{align}
    X_0(t,x,D_t,D_x)=D_t,\quad \text{and}\quad X_j(t,x,D_t,D_x)=X_j(x,D_x), \quad j=1,2,
\end{align}
with $X_1,X_2$ such that $X_1^*X_1(x,D_x)+{X_2}^*X_2(x,D)$ is an elliptic operator on $\mathbb{R}_x^2$. In other words, we are assuming the vector fields $iX_j(x,D)$ to be nondegenerate and linearly independent at any point $x\in\mathbb{R}^2_x$, so they form a global involutive distribution both in $\R^2$ and in $\R^3$. 
Under these assumptions the operator
\begin{align}
    P_1=D_t+X_1(x,D_x)\left(X_1^*X_1(x,D_x)+X_2^*X_2(x,D)\right)
\end{align}
is a prototype of a variable coefficient ZK-operator to which our solvability theorem applies. Hence, we have  the $L^2-L^2$ local solvability property for $P_1$ in $\R^{n+1}$ by Theorem \ref{thm.solv}.\\
\indent
Notice that the ''ellipticity in space'' of $X_1^*X_1(x,D_x)+X_2^*X_2(x,D)$ implies global involutivity, thus any prototype of this sort fits in our class. On the other hand, we assumed the ellipticity in space of $X_1^*X_1(x,D_x)+X_2^*X_2(x,D)$ just in order to represent a variable coefficient analogue of the ZK operator, and not because $P_1$-type operators need to satisfy this requirement. In fact, we can take more general cases, provided that the global involutivity is satisfied. 
In general, to describe a variable coefficient $P_1$-type ZK operator, one can take $X_1,X_2$ of the form
\begin{align}
 X_1(x,D_x)&=a_{11}(x)D_1+a_{12}(x)D_2,\quad a_{11}(x)\neq 0 \quad \text{or}\quad a_{12}(x)\neq 0\quad \forall x\in\mathbb{R}^2,\\
 X_2(x,D_x)&=a_{21}(x)D_1+a_{22}(x)D_2 \\
 [X_1,X_2](x,D_x)&=c_1(x)X_1(x,D_x)+c_2(x)X_2(x,D_x),\quad \quad  c_1,c_2\in C^\infty(\R^2;i\R),\quad \forall x\in\mathbb{R}^2_{x},
\end{align}
with $a_{ij}\in C^\infty(\R^2;\R)$ for all $i,j=1,2$, where the global involutivity is given by the third condition. 
In such a case our solvability theorem still applies.

\noindent\textbf{Non KdV-type operators.} To give a more complete picture of the objects of this paper, we provide immediate examples of non KdV-type operators of $P_1$-type. \\
\indent Let $N\leq n$ and $\R^{n+1}=\R_t\times\R_x$. Let also $\{iX_j \}_{j=1}^N$ be a global involutive system of real smooth vector fields in $\R^{n}_x$ (depending only on the space variables and on space derivatives), $X_1$ nondegenerate on $\R^{n+1}$, and $X_0\equiv 0$. Then
$$P_1(t,x,D_t,D_x)=P_1(x,D_x)=X_1\sum_{j=1}^NX_j^*X_j,$$
is of $P_1$-type both on $\R^{n}$ and on $\R^{n+1}$ but not of KdV-type on $\R^{n+1}$, since there is no time-evolution here. Note that we could also take $N\leq n+1$ and define an operator whose leading part has variable coefficients depending both on space and time.

\noindent\textbf{More degenerate variable coefficient KdV-type operators.} Still in the case $n=2$, one can  consider models where $N=1$, $iX_1(x,D_t,D_x)=X_1(x,D_x)$ is a nonvanishing real vector field with smooth coefficients, and $X_0=D_t$. In such a case we have no ellipticity in space for $X_1^*X_1$, since $X_1^*X_1$ cannot be elliptic under our assumption that $iX_1$ has real coefficients. Nevertheless, all the hypotheses, namely $(\mathrm{H1})$ and the nondegeneracy of $X_1$, are satisfied, hence our results apply - the Carleman and the solvability theorem - and the operator is locally solvable at any point of $\R^3$.\\
\indent This example can be generalized in any dimension $n\geq 2$. By taking $\{iX_j(x,D_x)\}_{j=1}^N$, $N\leq n$, being a globally involutive system of real smooth vector fields (indexed in such a way that $X_1$ is nondegenerate), and taking $iX_0=\partial_t$,  the corresponding operator $P_1$ is built from an operator $\mathcal{L}:=\sum_{j=1}^NX_j^*X_j$ which is not necessarily elliptic in space. 

In the rest of this section we will give concrete examples of KdV and non KdV $P_1$-type operators constructed via some Lie algebras, specifically stratified Lie algebras. These cases are taken into account due to their global involutive structure. One can, of course, design many other different examples.

\noindent \textbf{KdV-type operators built via the Heisenberg Lie algebra.} Let us start with the example mentioned in the introduction and related to the Heisenberg group. We restrict ourselves to $N=n=3$ - hence $P_1$ is defined on $\mathbb{R}^4_{t,x}$- which can be immediately generalized to the case $N=n=2k+1$, for any positive integer $k$. \\
\indent We take, as before, $X_0=D_t$ and $iX_j$, $j=1,2,3,$ as the vector fields giving the canonical basis of the Lie algebra of $\mathbb{H}^1$. Recall that the Lie algebra $\mathfrak{h}^1$ is stratified of step 2, and that $X_1,X_2$ are the so-called ''generators'' of the stratified Lie algebra. These vector fields, together with their commutator, generate the whole Lie algebra $\mathfrak{h}^1=\mathrm{Span}\{X_1,X_2,X_3\}$. In fact, $[X_1,X_2]=X_3$ and $[X_1,X_3]=[X_2,X_3]=0$ for all $x\in \mathbb{R}^3$. This guarantees that the vector fields $X_j$, $j=1,2,3$, form a global involutive structure. 
For completeness, we recall the expression of the vector fields below
\begin{align}
   i X_1(t,x,D_t,D_x)&= i X_1(x,D_x)=\partial_{x_1}-\frac{x_2}{2}\partial_{x_3},\\
 i X_2(t,x,D_t,D_x)&= i X_2(x,D_x)=\partial_{x_2}+\frac{x_1}{2}\partial_{x_3},\\
i X_3(t,x,D_t,D_x)&= i X_3(x,D_x)=\partial_{x_3}.
\end{align}
Note that $X_j=X_j^*$, thus $P_1=P_1^*$ in this case. Moreover, $X_j(x,D)$ is nondegenerate for all $j=0,\ldots, 3$, implying that the operators
$$P_{1,j_0}(t,x,D_t,D_x)=X_{j_0}(x,D_x)\sum_{j=1}^3X_j^2(x,D_x)+D_t,\quad j_0=1,2,3,\quad (t,x)\in\mathbb R^4,$$
are locally solvable at any point of $\R^4_{t,x}$ by Theorem \ref{thm.solv}. To define the same kind of operators in higher dimension it suffices to use the Heisenberg Lie algebra $\mathfrak{h}^k$ of dimension $2k+1$, and take $N=2k+1= n$.

\noindent \textbf{KdV-type operators built via the Heisenberg Lie algebra in higher dimensional spaces.} By taking the same vector fields as above generating $\mathfrak{h}^1$, one can cook up more singular operators in higher dimensional spaces. Take $N=3<n$, $\R^{n+1}=\R_t\times \R^3_x\times \R^{n-3}_y$, and $P_{1,j_0}$ as before, but now defined on $\R^{n+1}$, that is

$$P_{1,j_0}(t,x,y,D_t,D_x,D_y)=X_{j_0}(x,D_x)\sum_{j=1}^3X_j^2(x,D_x)+D_t,\quad j_0=1,2,3,\quad (t,x,y)\in\mathbb R_t\times \R^3_x\times\R^{n-3}_y.$$
Then we have that the system of vector fields $\{iX_j(x,D_x)\}_{j=1,2,3}$ is still globally involutive in $\R^{n+1}_{t,x,y}$, and the hypotheses of Theorem \ref{thm.solv} are satisfied. Hence, we get local solvability for these operators at any point of $\R^{n+1}_{t,x,y}$.\\
\indent Similar operators can be written by using the Heisenberg Lie algebra $\mathfrak{h}^k$, for every $k\geq 1$, and $N=2k+1<n$.

\noindent \textbf{KdV-type operators built via stratified Lie algebras: general construction.}
Following the previous constructions, we can formalize the procedure to build operators as above with any stratified Lie algebra of step $r$ on $\R^m$. Let  $\{iX_j\}_{j=1}^m$ be the canonical basis of $\mathfrak{g}$ (via the expnential map), and
\begin{align}
    \mathfrak{g}&=\oplus_{j=1}^r \mathfrak{g}_j,\quad [\mathfrak{g}_j,\mathfrak{g}_k]\subset\mathfrak{g}_{k+j},
\end{align}
where $\mathfrak{g}_1$ generates $\mathfrak{g}$ as an algebra, that is, linear combinations of the elements of $\mathfrak{g}_1$ and of their iterated commutators up to length $r$, generate the whole $\mathfrak{g}$.
In particular, one has that
$$n_j:=\mathrm{dim}(\mathfrak{g}_j)=\mathrm{dim}(\mathrm{Span}\{iX_1(x,D_x),\ldots, iX_{n_j}(x,D_x) \})$$
is such that
$$\sum_{j=1}^r n_j=m.$$
Let now $N=m\leq n$, $\R^{n+1}=\R_t\times \R^m_x\times \R^{n-m}_y$, $j_0=1,\ldots,m,$ and
\begin{align}
    P_{1,j_0}(t,x,y,D_t,D_x,D_y)=X_{j_0}(x,D_x)\sum_{j=1}^m X_j^*(x,D_x)X_j(x,D_x)+D_t,\quad \,\,(t,x,y)\in \R_t\times \R^3_x\times\R^{n-3}_y.
\end{align}
Then, once again, we have that $\{iX_j\}_{j=1}^m$ is a global involutive system of vector fields on $\R^{n+1}$. Moreover, the vector fields, being elements of the canonical basis of the Lie algebra, are all nondegenerate in $\R^{n+1}$. Hence, for every $j_0=1,\ldots, m$, the operator $ P_{1,j_0}$ satisfies the hypotheses of Theorem \ref{thm.solv}, leading to the local solvability property for all these models.

\noindent\textbf{Non KdV-type operators built via Lie algebras.} 
One can get immediate examples of non KdV-type operators just by taking the previous examples with $X_0\equiv 0$. 
 
\begin{remark}
In our examples above we have focused primarily on space-variable coefficient KdV-type operators. We want to stress that several KdV-type operators, with coefficients depending both on $(t,x)$ or just on $t$,  are still included in our class $P_1$.
\end{remark}
\section{Final remarks and applications to dispersive equations}

In this conclusive section we would like to discuss possible applications in the context of dispersive equations of KdV-type, as well as some open connected questions.

\noindent \textbf{Uniqueness problems.} \label{sec.finalRem-Applic}
Uniqueness problems, or Hardy uncertainty principles (see \cite{EKPV4}), for dispersive equations, are strictly related to Carleman estimates. 

In general, when dealing with dispersive equations, one aims at global in space and local or global in time results. Global in time results are more challenging an not always attainable, so it is a natural common rule to investigate local in time properties first, and, afterwards, to extend, if possible, the result at all times.

The same considerations apply to Carleman estimates. This means that one needs global in space and local in time estimates if one wants to apply them in the context of dispersive equations. This is usually done via a combination of strategies. One of this strategies consists in considering the operator with bounded space-variable coefficients with a certain decay property. This condition is actually not just a technical condition, but is also related to the validity of smoothing properties for dispersive operators (see, for instance, \cite{Doi},\cite{KPRV}, \cite{FS0}, \cite{RZ},\cite{MMT}). Once the Carleman estimate is made global (in space) through a series of conditions in spirit as the one above, then one can attain uniqueness results via a precise cutting off procedure according to the assumptions considered on the solutions.

That said, it should be possible to reach global (in all but at most one direction) Carleman estimates for $P_1$ and $P_1^*$ from our local estimates in Theorem \ref{thm.Carleman}. These global estimates can be applied to study uniqueness problems for dispersive operators belonging to our class.  For what we just explained above, by introducing suitable conditions on the coefficients of the operator and on the weight function $f$, one should be able to produce the desired global inequality. However, even if the road map to get a global inequality is in the proof of Theorem \ref{thm.Carleman}, its derivation and the application to uniqueness problems is nontrivial. \\
\indent Let us also stress that, if one is interested in the study of local unique continuation properties, then our estimates are already strong enough to pursue this goal, and one should focus more on the study of the geometric conditions  leading to the result, that is, in other words, on the  choice of the suitable weight function $f$.

\noindent\textbf{Other KdV-type models and relative questions.}
 In our discussion, we have motivated our analysis with the connection of the class $P_1$ with KdV-type operators. 
  However, it must be said that we have a rigorous derivation of the ZK equation as a sort of generalization of the KdV one in (space) dimension 2, but for the higher dimensional case the argument is much harder and not established to the author knowledge. On the other hand, it is known that third order equations play a central role in water waves, nonlinear optic and related fields. A detailed interesting description of third order dispersive equations and of their physical role is contained in \cite{KS}, where a class of third order operators with constant coefficients is studied. Let us remark that  a subclass of the class in \cite{KS} of constant coefficient operators is included in our general variable coefficient class $P_1$. That said, it would be interesting to investigate Carleman estimates, local solvability, and connected problems, for operators in the very general form
  \begin{align}
      P_2(x,D):=\sum_{j=1}^{M}\sum_{k=1}^{N}X_jX_k^*X_k+X_0, \quad 1\leq M\leq N, x\in\R^{n+1}.
  \end{align}
Note that this class, which generalizes and contains $P_1$, clearly still contains KdV-type operators,  both with constant and with variable coefficients, as well as the whole class studied in \cite{KS}.  Due to these considerations, it would be interesting to understand weather $P_2$ is the right generalization of KdV-type operators, or if $P_1$ provides the best description, being clear that $P_2$ is a class wider than $P_1$.
In any case, the validity of Carleman estimates in this generality is still an open problem.


\bibliographystyle{plain}

\end{document}